\documentclass{amsart}
\usepackage{amsmath,amssymb,amscd,latexsym}

\usepackage[all]{xy}
\newcommand{\A}{{\mathbb{A}}}
\newcommand{\C}{{\mathbb{C}}}

\newcommand{\Q}{{\mathbb{Q}}}

\newcommand{\Pro}{{\mathbb{P}}}

\newcommand{\R}{{\mathbb{R}}}
\newcommand{\Z}{{\mathbb{Z}}}

\newcommand{\Zln}{{\mathbb{Z}}/\ell^{\nu}}
\newcommand{\hZ}{\hat{{\mathbb{Z}}}}
\newcommand{\hFab}{\hat{F}_{\mathrm{ab}}}

\newcommand{\car}{\mathrm{char}\;}
\newcommand{\et}{\mathrm{\acute{e}t}}

\newcommand{\ok}{\bar{k}}

\newcommand{\pro}{\mathrm{pro}}

\newcommand{\colim}{\operatorname*{colim}}
\newcommand{\hocolim}{\operatorname*{hocolim}}
\newcommand{\holim}{\operatorname*{holim}}
\newcommand{\Et}{\mathrm{Et}\,}
\newcommand{\hEt}{\hat{\mathrm{Et}}\,}
\newcommand{\chEt}{\mathrm{c}\hat{\mathrm{Et}}\,}

\newcommand{\compl}{\hat{(\cdot)}}

\newcommand{\Gal}{\mathrm{Gal}}

\newcommand{\Hom}{\mathrm{Hom}}

\newcommand{\homg}{\mathrm{hom}_G}

\newcommand{\map}{\mathrm{map}}

\newcommand{\sk}{\mathrm{sk}}

\newcommand{\Sm}{\mathrm{Sm}}

\newcommand{\Tot}{\mathrm{Tot}}
\newcommand{\Ch}{{\mathcal C}}

\newcommand{\Eh}{{\mathcal E}}
\newcommand{\hEh}{\hat{\mathcal E}}
\newcommand{\hEhg}{\hat{{\mathcal E}}_G}
\newcommand{\Fh}{{\mathcal F}}
\newcommand{\Gh}{{\mathcal G}}
\newcommand{\Ghg}{{\mathcal G}_G}
\newcommand{\Lhg}{{\mathcal L}_G}
\newcommand{\Thg}{{\mathcal T}_G}
\newcommand{\Hh}{{\mathcal H}}

\newcommand{\hHh}{\hat{{\mathcal H}}}

\newcommand{\hHhg}{\hat{{\mathcal H}}_G}

\newcommand{\Lh}{{\mathcal L}}
\newcommand{\Mh}{{\mathcal M}}

\newcommand{\Rh}{{\mathcal R}}
\newcommand{\Sh}{{\mathcal S}}

\newcommand{\SHh}{{\mathcal{SH}}}

\newcommand{\hSHh}{{\hat{\mathcal{SH}}}}
\newcommand{\hSHhg}{{\hat{\mathcal{SH}}_G}}

\newcommand{\hSh}{\hat{\mathcal S}}
\newcommand{\hShg}{\hat{\mathcal S}_G}

\newcommand{\hShpg}{\hat{\mathcal S}_{\ast,G}}

\newcommand{\Spg}{\mathrm{Sp}(\Sh_{\ast,G})}
\newcommand{\hSpg}{\mathrm{Sp}(\hShpg)}

\newcommand{\hShp}{\hat{\mathcal S}_{\ast}}

\newcommand{\hSp}{\mathrm{Sp}(\hShp)}

\newcommand{\Th}{{\mathcal T}}
\newcommand{\Xh}{\mathcal{X}}

\newcommand{\oX}{\bar{X}}

\newcommand{\hMU}{\hat{M}U}

\newtheorem{theorem}{Theorem}[section]
\newtheorem{lemma}[theorem]{Lemma}
\newtheorem{prop}[theorem]{Proposition}
\newtheorem{defn}[theorem]{Definition}

\newtheorem{cor}[theorem]{Corollary}
\newtheorem{conjecture}[theorem]{Conjecture}
\theoremstyle{definition}

\newtheorem{remark}[theorem]{Remark}
\begin{document}
\title{Continuous group actions on profinite spaces}
\author{Gereon Quick}
\date{}
\begin{abstract}
For a profinite group, we construct a model structure on profinite spaces and profinite spectra with a continuous action. This yields descent spectral sequences for the homotopy groups of homotopy fixed point spaces and for stable homotopy groups of homotopy orbit spaces. Our main example is the Galois action on profinite \'etale topological types of varieties over a field. One motivation is to understand Grothendieck's section conjecture in terms of homotopy fixed points.
\end{abstract}
\maketitle
\section{Introduction}

Let $\hSh$ be the category of profinite spaces, i.e simplicial objects in the category of profinite sets.  Examples of profinite spaces arise in algebraic geometry. For a locally noetherian scheme $X$, we denote by $\hEt X$ the profinitely completed \'etale topological type of $X$ of Friedlander, see \cite{fried} and \cite{profinhom}. It is a profinite space that collects the information of the \'etale topology on the scheme. Now let $k$ be a field, $\ok$ a separable closure of $k$ and $G_k$ its Galois group over $k$. Let $X$ be a variety over $k$ and let $\oX=X\otimes_k \ok$ be the base change of $X$ to $\ok$. Since $G_k$ acts on $\oX$ and since $\hEt$ is a functor, there is an induced $G_k$-action on $\hEt \oX$. This Galois action is an important property of the \'etale topological type of $\oX$. Furthermore, there is a natural sequence of profinite spaces
\begin{equation}\label{fibersequence}
\hEt \oX \longrightarrow \hEt X \longrightarrow \hEt k.
\end{equation}
As $\hEt k \simeq BG_k$, this inspires one to think of $\hEt X$ as the homotopy orbit of $\hEt \oX$ under its $G_k$-action. In fact, this would generalize a theorem of Cox' on real algebraic varieties in \cite{cox} that there is a weak equivalence of pro-spaces $\Et X \simeq X(\C)\times_G EG$ for $G=\Gal(\C/\R)$. 
 
The main purpose of this paper is to provide a rigid framework for the Galois action on \'etale topological types via model categories in which a generalization of Cox' result and its applications can be proven. In particular, we are going to construct homotopy fixed points of \'etale topological types of varieties over arbitrary base fields. 

So let us describe the general setup and thereby outline the content of the paper. Let $G$ be an arbitrary profinite group and let $\hShg$ be the category of profinite spaces with $G$ acting continuously in each level and equivariant face and degeneracy maps. Various model structures on $\hSh$ have been constructed, see \cite{ensprofin} and \cite{profinhom}. We construct a left proper fibrantly generated model structure on $\hShg$ such that the weak equivalences (fibrations) in $\hShg$ are the maps that are weak equivalences (resp. fibrations) in $\hSh$ and the cofibrant objects are the profinite spaces with a free $G$-action. In order to do so, one is tempted to use an adjoint functor argument as Goerss did in \cite{goerss}, but the subtle point is that the mapping space functor $\hom_{\hSh}(G,-):\hSh \to \hShg$ is not the natural right adjoint of the forgetful functor $\hSh \to \hShg$ since the sets $\Hom_{\hSh}(X,Y)$ do not have to be profinite for general $X$ and $Y$ in $\hSh$. Therefore, encouraged by the referee of this paper, we give a direct proof with explicit generating fibrations and trivial fibrations.\footnote{The initial model structure and the argument in the proof were closer to the one of \cite{goerss} using an intermediate strict model structure followed by a Bousfield localization.}\\
Let $EG$ denote the universal profinite covering space of the classifying space $BG$. A convenient point of profinite spaces is that $EG$ and $BG$ are objects of $\hSh$. Hence the homotopy fixed points of a profinite $G$-space $X$ can be defined as the simplicial mapping space $X^{hG}=\homg(EG, RX)$ of continuous $G$-equivariant maps, where $RX$ denotes a functorial fibrant replacement in $\hShg$. Moreover, there is a descent spectral sequence for homotopy groups of homotopy fixed points of connected pointed profinite $G$-spaces
$$E_2^{s,t}=H^s(G;\pi_tX) \Rightarrow \pi_{t-s}(X^{hG}).$$
The $E_2$-term of this spectral sequence is continuous cohomology of the profinite groups $\pi_tX$, where $\pi_1X$ might be a nonabelian profinite group.\\ 
For the homotopy orbit space $X_{hG}=X\times_G EG$, we construct a spectral sequence computing the homology $H_{\ast}(X_{hG};M)$ for any profinite abelian group $M$. Both for homotopy fixed points and homotopy orbit spaces, the construction of the spectral sequence follows naturally from the work of Bousfield-Kan \cite{bouskan}.\\
There is also a stable homotopy category $\hSHhg$ for $G$-spaces using profinite $G$-spectra. The well-known machinery yields a homotopy orbit spectral sequence for stable profinite homotopy groups. This generalizes the notion of pro-$f$-spectra of Davis \cite{davis}.\\
In \cite{hfplt}, we study homotopy fixed points of profinite spectra with a continuous $G$-action in more detail. The main application is to provide a natural setting for the continuous action of the extended Morava stabilizer group $G_n$ on Lubin-Tate spectra $E_n$. Since $G_n$ acts continuously on the profinite homotopy groups $\pi_kE_n$, it seems natural to study the spectra $E_n$ as profinite spectra. The construction of a descent spectral sequence for the homotopy fixed point spectra $E_n^{hG_n}$, respectively $E_n^{hG}$ for any closed subgroup of $G_n$, then follows easily in the category of profinite spectra without using the results of \cite{devinatzhopkins}. These methods provide, in particular, a new construction for homotopy fixed points under open subgroups of $G_n$.

In the last section we return to the situation of a Galois group $G_k$ acting on the variety $\oX$. We will prove the following generalization of Cox's theorem mentioned above.
\begin{theorem}\label{etalehomorbitintro}
Let $k$ be a field with absolute Galois group $G_k$ and let $X$ be a geometrically connected variety over $k$. Then the canonical map 
$$(\lim_L \hEt X_L) \times_{G_k} EG_k \to \hEt X$$
is a weak equivalence of profinite spaces, where the limit is taken over all finite Galois extensions $L/k$ in $\ok$ and $X_L$ denotes $X\otimes_k L$.
\end{theorem}
We would like to prove the theorem directly for $\hEt \oX$ but it is not clear that $\hEt \oX$ is an object of the category of profinite $G_k$-spaces defined above, i.e. that the action of $G_k$ remains level-wise continuous after taking limits over hyerpcoverings. But the canonical $G_k$-equivariant map $\hEt \oX \longrightarrow \lim_L \hEt X_L$ is a weak equivalence of profinite spaces and the latter space has all the properties we need. In particular, for a variety over a field, the theorem provides the following intuition with a precise meaning: The two perspectives of viewing the \'etale topological type of $X$ as a profinite space $\hEt X$ over $\hEt k \simeq BG_k$ or as a profinite space $\hEt \oX$ together with its induced $G_k$-action are essentially equivalent.\\
By Theorem \ref{etalehomorbitintro}, the homotopy orbit spectral sequence above may be written as a Galois descent spectral sequence for stable profinite \'etale homotopy groups of $X$:
$$E^2_{p,q}=H_p(G_k;\pi^{\et,s}_q(\oX)) \Rightarrow \pi^{\et,s}_{p+q}(X).$$
Moreover, we will show that the stable \'etale realization functor can be viewed as a functor from motivic spectra to $\hSHh_{G_k}$. Finally, we show that a refined version of \'etale cobordism of \cite{etalecob} satisfies Galois descent in the sense that for any variety $X$ over $k$ there is a spectral sequence 
$$E_2^{s,t}=H^s(G_k;\hMU_{\et}^{t}(\oX)) \Rightarrow \hMU_{\et}^{s+t}(X).$$
An $\ell$-adic version of \'etale (co)bordism has been used in \cite{cycles} in order to study the integral cycle map from algebraic cycles to \'etale homology for schemes over an algebraically closed base field. The descent spectral sequence above should be useful for a future application of the techniques of \cite{cycles} for varieties over finite fields.

But the main motivation for studying homotopy fixed points under the Galois action is Grothendieck's section conjecture. In fact, the conjecture can be formulated as an isomorphism on $k$-rational points $X(k)$ and $G_k$-homotopy fixed points of $\hEt \oX$. Namely, as the curves involved in the conjecture are $K(\pi,1)$-varieties, we get an isomorphism for continuous cohomology
$$
\Hom_{\hHh/\hEt k}(\hEt k,\hEt X) \cong \Hom_{\mathrm{out},G_k}(G_k,\pi_1X),$$
where the right-hand side denotes outer homomorphisms that are compatible with the projection to $G_k$. Moreover, with the homotopy equivalence $\hEt k \simeq BG_k$ and Theorem \ref{etalehomorbitintro} we get by adjunction an isomorphism 
\begin{equation}\label{scintro}
\Hom_{\hHh/\hEt k}(\hEt k,\hEt X) \cong \Hom_{\hHh_{G_k}}(EG_k,\hEt \oX)\cong \pi_0(\hEt \oX)^{hG_k}.\end{equation}
So one could read the section conjecture as the conjecture that the canonical map from $X(k)$ to the homotopy fixed point set on the right-hand side of (\ref{scintro}) is bijective.\\ 
This point of view might be of interest as analogues of the section conjecture over the reals, shown by Mochizuki in  \cite{mochizuki2}, could be proved by Pal using homotopy fixed points results \cite{pal1}. Over $\R$, Cox' theorem is a crucial point in the proof of the section conjecture. It is likely, that its generalization, Theorem \ref{etalehomorbitintro} above, is useful for an extension of the methods over $\R$ to fields finitely generated over $\Q$ or $\Q_p$.\\
The key new progress of this paper for this direction is that Theorem \ref{etalehomorbitintro} allows the reinterpretation (\ref{scintro}) and that, since $B\pi_1\oX$ and $\hEt \oX$ are naturally profinite spaces, the machinery described above provides a good notion of homotopy fixed points of these spaces. This opens a new homotopy theoretical tool kit to analyze the section conjecture.

{\bf Acknowledgements}: Apart from the motivation by \'etale homotopy theory, the starting point of this project was a hint by Dan Isaksen that the profinite spectra of \cite{etalecob} should fit well in the picture for Lubin-Tate spectra. I would like to thank him very much to share this idea with me. I would like to thank Fabien Morel for a discussion on \'etale homotopy types. I am grateful to Kirsten Wickelgren, Daniel Davis, Johannes Schmidt and Mike Hopkins for helpful comments. Moreover, I would like to express my gratitude to the kind referee of this paper whose suggestions and comments helped to clarify and improve the content. Finally, I am grateful to the Institute for Advanced Study in Princeton for its support and hospitality and the inspiring atmosphere in which the final version of this paper has been written. 
\section{Homotopy theory of profinite $G$-spaces}

\subsection{Profinite spaces}
First we recall some basic notions for profinite spaces and their homotopy category from \cite{ensprofin} and \cite{profinhom}. For a category $\Ch$ with small limits, the pro-category of $\Ch$, denoted pro-$\Ch$, has as objects all cofiltering diagrams $X:I \to \Ch$. Its sets of morphisms are defined as
$$\Hom_{\pro-\Ch}(X,Y):=\lim_{j\in J}\colim_{i\in I} \Hom_{\Ch}(X_i,Y_j).$$
A constant pro-object is indexed by the category with one object and one identity map. The functor sending an object $X$ of $\Ch$ to the constant pro-object with value $X$ makes $\Ch$ a full subcategory of pro-$\Ch$. The right adjoint of this embedding is the limit functor $\lim$: pro-$\Ch$ $\to \Ch$, which sends a pro-object $X$ to the limit in $\Ch$ of the diagram corresponding to $X$.\\
Let $\Eh$ denote the category of sets and let $\Fh$ be the full subcategory of finite sets. Let $\hEh$ be the category of compact Hausdorff totally disconnected topological spaces. We may identify $\Fh$ with a full subcategory of $\hEh$ in the obvious way. The limit functor $\lim$: pro-$\Fh \to \hEh$ is an equivalence of categories.\\
We denote by $\hSh$ (resp. $\Sh$) the category of simplicial profinite sets (resp. simplicial sets). The objects of $\hSh$ (resp. $\Sh$) will be called {\em profinite spaces} (resp. {\em spaces}). The forgetful functor $\hEh \to \Eh$ admits a left adjoint $\compl:\Eh \to \hEh$. It induces a functor $\compl:\Sh \to \hSh$, which is called {\em profinite completion}. It is left adjoint to the forgetful functor $|\cdot|:\hSh \to \Sh$ which sends a profinite space to its underlying simplicial set.\\ 
For a profinite space $X$ we define the set $\Rh(X)$ of simplicial open equivalence relations on $X$. An element $R$ of $\Rh(X)$ is a simplicial profinite subset of the product $X\times X$ such that, in each degree $n$, $R_n$ is an equivalence relation on $X_n$ and an open subset of $X_n\times X_n$. It is ordered by inclusion. For every element $R$ of $\Rh(X)$, the quotient $X/R$ is a simplicial finite set and the map $X \to X/R$ is a map of profinite spaces. The canonical map $X \to \lim_{R\in \Rh(X)} X/R$ is an isomorphism in $\hSh$, cf. \cite{ensprofin}, Lemme 1.\\ 
Let $X$ be a profinite space. The continuous cohomology $H^{\ast}(X;\pi)$ of $X$ with coefficients in the topological abelian group $\pi$ is defined as the cohomology of the complex $C^{\ast}(X;\pi)$ of continuous cochains of $X$ with values in $\pi$, i.e. $C^n(X;\pi)$ denotes the set $\Hom_{\hat{\Eh}}(X_n,\pi)$ of continuous maps $\alpha:X_n \to \pi$ and the differentials $\delta^n:C^n(X;\pi)\to C^{n+1}(X;\pi)$ are the morphisms associating to $\alpha$ the map $\sum_{i=0}^{n+1}\alpha \circ d_i$, where $d_i$ denotes the $i$th face map of $X$, see \cite{tate} and \cite{ensprofin}. If $\pi$ is a finite abelian group and $Z$ a simplicial set, then the cohomologies $H^{\ast}(Z,\pi)$ and $H^{\ast}(\hat{Z},\pi)$ are canonically isomorphic.\\
If $\Gamma$ is an arbitrary profinite group, we may still define the first cohomology of $X$ with coefficients in $\Gamma$ as done by Morel in \cite{ensprofin}, p. 355. The functor 
$$X\mapsto \Hom_{\hEh}(X_0,\Gamma)$$
is represented in $\hSh$ by a profinite space $E\Gamma$. We define the $1$-cocycles $Z^1(X;\Gamma)$ to be the set of continuous maps $f:X_1 \to \Gamma$ such that $f(d_0x)f(d_2x)=f(d_1x)$ for every $x \in X_1$. The functor $X\mapsto Z^1(X;\Gamma)$ is represented by a profinite space $B\Gamma$. Explicit constructions of $E\Gamma$ and $B\Gamma$ may be given in the standard way as in $\Sh$. Furthermore, there is a map $\delta:\Hom_{\hSh}(X,E\Gamma) \to Z^1(X;\Gamma)\cong \Hom_{\hSh}(X,B\Gamma)$ which sends $f:X_0 \to \Gamma$ to the $1$-cocycle $x\mapsto \delta f(x)=f(d_0x)f(d_1x)^{-1}$. We denote by $B^1(X;\Gamma)$ the image of $\delta$ in $Z^1(X;\Gamma)$ and we define the pointed set $H^1(X,\Gamma)$ to be the quotient $Z^1(X;\Gamma)/B^1(X;\Gamma)$. Finally, if $X$ is a profinite space, we define $\pi_0X$ to be the coequalizer in $\hEh$ of the diagram $d_0,d_1:X_1 \rightrightarrows X_0$.\\
The profinite fundamental group of $X$ is defined via covering spaces. There is a universal profinite covering space $(\tilde{X},x)$ of $X$ at a vertex $x \in X_0$. Then $\pi_1(X,x)$ is defined to be the group of automorphisms of $(\tilde{X},x)$ over $(X,x)$. It has a natural structure of a profinite group as the limit of the finite automorphism groups of the finite Galois coverings of $(X,x)$. The collection of the $\pi_1(X,x)$ for all $x\in X_0$ defines a profinite fundamental groupoid $\Pi X$. A profinite local coefficient system $\Mh$ on $X$ is a functor from $\Pi X$ to profinite abelian groups. The cohomology of $X$ with coefficients in $\Mh$ is then defined as the cohomology of the complex $\hom_{\Pi X}(\tilde{X}_*,\Mh)$ of continuous natural transformations. For any further details, we refer the reader to \cite{profinhom}. The relation of this $\pi_1(X,x)$ to the usual fundamental group of a simplicial set is described by the following result.
\begin{prop}\label{profinitecompletion}
For a pointed simplicial set $X$, the canonical map from the profinite group completion of $\pi_1(X)$ to $\pi_1(\hat{X})$ is an isomorphism, i.e. $$\widehat{\pi_1(X)}\cong \pi_1(\hat{X})$$ 
as profinite groups. 
\end{prop}    
\begin{defn}\label{defnwe}
A morphism $f:X\to Y$ in $\hSh$ is called\\
1) a {\em weak equivalence} if the induced map $f_{\ast}:\pi_0(X) \to \pi_0(Y)$ is an isomorphism of profinite sets, $f_{\ast}:\pi_1(X,x) \to \pi_1(Y,f(x))$ is an isomorphism of profinite groups for every $x\in X_0$ and $f^{\ast}:H^q(Y,\Mh) \to H^q(X,f^{\ast}\Mh)$ is an isomorphism for every local coefficient system $\Mh$ of finite abelian groups on $Y$ for every $q\geq 0$;\\
2) a {\em cofibration} if $f$ is a level-wise monomorphism;\\
3)a {\em fibration} if it has the right lifting property with respect to every cofibration that is also a weak equivalence. 
\end{defn}
These classes of morphisms fit into a simplicial fibrantly generated left proper model structure on $\hSh$. For every natural number $n\geq 0$ we choose a finite set with $n$ elements, e.g. the set $\{0,1, \ldots, n-1\}$, as a representative of the isomorphism class of sets with $n$ elements. We denote the set of these representatives by $\Th$. Moreover, for every isomorphism class of finite groups, we choose a representative with underlying set $\{0,1,\ldots,n-1\}$. Hence for each $n$ we have chosen as many groups as there are relations on the set $\{0,1,\ldots,n-1\}$. This ensures that the collection of these representatives forms a set which we denote by $\Gh$. Finally, we denote by $\Lh$ the collection of representatives of isomorphism class of finite abelian modules $M \in \Gh$ with an action of a finite group $\Gamma \in \Gh$. By the definition of $\Gh$, it is clear that $\Lh$ forms a set as well.\\
Let $\Gamma$ be a profinite group and let $\hSh/B\Gamma$ denote the category of profinite spaces equipped with a map to $B\Gamma$. Moreover, let $\hSh_{\Gamma}$ be the category of profinite spaces with a levelwise continuous $\Gamma$-action. We will discuss this category in more detail in the next section. There is a functor $\hSh_{\Gamma} \to \hSh/B\Gamma$ sending $Y$ to the Borel construction $E\Gamma \times_{\Gamma} Y \to B\Gamma$. On the other hand, there is the functor $\hSh/B\Gamma \to \hSh_{\Gamma}$ sending $X \to B\Gamma$ to the $\Gamma$-principal fibration $E\Gamma\times_{B\Gamma}X \to X$. These two functors form a pair of adjoint functors, cf. \cite{gj}, VI Lemma 4.6, i.e. there is a natural bijection
\begin{equation}\label{adjoint}
\Hom_{\hSh_{\Gamma}}(E\Gamma\times_{B\Gamma}X,Y)\cong \Hom_{\hSh/B\Gamma}(X,E\Gamma \times_{\Gamma} Y).
\end{equation}
%
For a profinite $\Gamma$-module $M$, we denote the profinite space $E\Gamma \times_{\Gamma} K(M,n)$ by $K^{\Gamma}(M,n)$ and similarly $E\Gamma \times_{\Gamma} L(M,n)$ by $L^{\Gamma}(M,n)$. 
We define two sets $P$ and $Q$ of morphisms in $\hSh$ as follows:
$$\begin{array}{lll}
P & \mathrm{consisting~of} & E\Gamma \to B\Gamma, B\Gamma \to \ast,~ L^{\Gamma}(M,n)\to K^{\Gamma}(M,n+1),\\
   &   &  K^{\Gamma}(M,n)\to B\Gamma,~K(S,0)\to K(S,0)\times K(S,0)\\
   &   & K(S,0) \to \ast ~ \mathrm{for~every~finite~set}~S\in \Th,\\
    &  & \mathrm{every~finite~abelian}~\Gamma-\mathrm{module}~M\in \Lh,\\
   &   & \mathrm{every~finite~group}~\Gamma \in \Gh, ~\mathrm{and~every}~ n\geq 0;\\
Q & \mathrm{consisting~of} & E\Gamma \to \ast,~ L^{\Gamma}(M,n) \to B\Gamma~ \mathrm{for~every}~\Gamma \in \Gh,\\
& &  \mathrm{every}~M\in \Lh ~\mathrm{and~every}~ n\geq 0. 
\end{array}$$
In \cite{profinhom}, the following theorem had already been claimed, but the classes $P$ and $Q$ had been chosen too small. The kind referee of the present paper pointed this out and also gave a suggestion how to correct the error. We are very grateful for this hint. 
\begin{theorem}\label{modelstructure}
The above defined classes of weak equivalences, cofibrations and fibrations provide $\hSh$ with the structure of a fibrantly generated left proper model category with $P$ the set of generating fibrations and $Q$ the set of generating trivial fibrations. We denote the homotopy category by $\hHh$.
\end{theorem}
\begin{proof}
We show that there is a fibrantly generated model structure by checking the four conditions of the dual of Kan's Theorem 11.3.1 in \cite{hirsch}. It is clear that the weak equivalences satisfy the 2-out-of-3 property and are closed under retracts. We denote by $P$-cocell the subcategory of relative $P$-cocell complexes consisting of limits of pullbacks of elements of $P$. We write $P$-proj for the maps having the left lifting property with respect to all maps in $P$ and $P$-fib for the maps having the right lifting property with respect to all maps in $P$-proj. And we do so similarly for $Q$. Now we check the remaining hypotheses of Kan's Theorem 11.3.1 in \cite{hirsch}.\\ 
1. We have to show that the codomains of the maps in $P$ and $Q$ are cosmall relative to $P$-cocell and $Q$-cocell, respectively. This is clear for the terminal object $\ast$. The proof of Theorem 2.11 of \cite{profinhom} shows that $B\Gamma$ is cosmall relative to $P$. We check this now for the objects $K^{\Gamma}(M,n)$. The case of $K(S,0)\times K(S,0)$ is similar. By definition of cosmallness we have to show that the canonical map 
$$\varphi:\colim_{\alpha} \Hom_{\hSh}(Y_{\alpha},K^{\Gamma}(M,n)) \to \Hom_{\hSh}(\lim_{\alpha}Y_{\alpha},K^{\Gamma}(M,n))$$ is an isomorphism for some cardinal $\kappa$, where $Y_{\alpha}$ is any projective system whose indexing category is of cardinality $\kappa$. By the definition of the spaces $K^{\Gamma}(M,n)$ and the adjunction (\ref{adjoint}), this map is equal to the map 
$$\colim_{\alpha} \bigcup_{f_{\alpha}:Y_{\alpha}\to B\Gamma} Z_{\Gamma}^n(E\Gamma \times_{B\Gamma,f_{\alpha}} Y_{\alpha};M) \to 
\bigcup_{f:\lim Y_{\alpha} \to B\Gamma}Z_{\Gamma}^n(E\Gamma \times_{B\Gamma,f} \lim_{\alpha}Y_{\alpha};M),$$ 
where $Z^n_{\Gamma}(E\Gamma \times_{B\Gamma,f} Y,M)$ denotes the subgroup of the $n$th cocycles in the complex of $\Gamma$-equivariant cochains $C^n_{\Gamma}(E\Gamma \times_{B\Gamma} Y;M):=\Hom_{\hEh_{\Gamma}}(\Gamma \times Y_n, M)$. The unions are taken over all maps $f:Y\to B\Gamma$, where we use the cosmallness of $B\Gamma$ to deduce $\colim_{\alpha}\Hom_{\hSh}(Y_{\alpha},B\Gamma)=\Hom_{\hSh}(\lim_{\alpha} Y_{\alpha},B\Gamma)$, cf. \cite{profinhom}, Theorem 2.11. Since $\hEh_{\Gamma}$ is equivalent to the pro-category of finite $\Gamma$-sets and since all objects in a pro-category are cosmall by \cite{prospectra} Corollary 3.5, we see that the map 
$$\colim_{\alpha} \Hom_{\hEh_{\Gamma}}(\Gamma \times Y_{\alpha,n}, M) \to \Hom_{\hEh_{\Gamma}}(\Gamma \times \lim_{\alpha}Y_{\alpha,n},M)$$ 
is an isomorphism. Hence $\varphi$ is an isomorphism.\\ 
2. We have to show that every $Q$-fibration is both a $P$-fibration and a weak equivalence. First we observe that every map in $Q$ is a composition of maps in $P$ and hence a relative $P$-cocell complex. By \cite{hoveybook},  Lemma 2.1.10, this implies that every map in $P$-proj is an element of $Q$-proj and hence also $Q$-fib $\subset$ $P$-fib.\\
Furthermore, if $f$ is a monomorphism in each dimension, then $f$ is an element of $Q$-proj by Lemma 2.8 of \cite{profinhom}. Hence $Q$-fib is contained in the class of maps that have the right lifting property with respect to all monomorphisms. By Lemme 3 of \cite{ensprofin} this implies that the maps in $Q$-fib are simplicial homotopy equivalences and hence weak equivalences.\\
3. We have seen that every $P$-projective map is a $Q$-projective map. It remains to show $P$-proj $\subseteq W$. So let $f:X\to Y$ be in $P$-proj. The left lifting property with respect to $K(S,0) \to \ast$ for any finite set $S$, shows that the map $H^0(f;S)$ is surjective and the left lifting property with respect to all maps $K(S,0)\to K(S,0)\times K(S,0)$ shows that any two liftings $Y\to K(S,0)$ agree. Thus $H^0(f;S)$ is an isomorphism. We deduce that $\pi_0(f)$ is a bijection.\\ 
So from now on we can assume that $X$ and $Y$ are connected. Since $f$ is in $P$-proj, we deduce as in the proof of Lemma 2.8 of \cite{profinhom} that the maps $f^{\ast}:Z^1(Y;\Gamma)\to Z^1(X;\Gamma)$ and $C^0(Y;\Gamma)\to C^0(X;\Gamma)\times_{Z^{1}(X;\Gamma)}Z^{1}(Y;\Gamma)$ are surjective. The latter implies that if, for an element $\beta \in Z^1(Y,\Gamma)$, $f^{\ast}(\beta)$ is a boundary, then $\beta$ is itself a boundary. Hence $f$ induces an isomorphism $f^1:H^1(Y;\Gamma) \cong H^1(X;\Gamma)$.\\
Let $\Mh$ be a local coefficient system of finite abelian groups on $\Pi Y$, the profinite fundamental groupoid of $Y$. Since $X$ and $Y$ are connected, after choosing a vertex $y\in Y_0$ in the image of $f$ and a vertex $x\in X_0$ such that $f(x)=y$, there are equivalences of groupoids $\Pi X\simeq \pi_1(X,x)$ and $\Pi Y \simeq \pi_1(Y,y)$ compatible with the morphisms induced by $f$, cf. \cite{gj}, VI, proof of Lemma 3.9. We set $\pi_X:=\pi_1(X,x)$ and $\pi_Y:=\pi_1(Y,y)$. Under these identifications, the local system $\Mh$ corresponds to a finite abelian $\pi_Y$-module $M$. Since the action of $\pi_Y$ is continuous on $M$, it factors through a finite quotient of $\pi_Y$ which is isomorphic to some $\Gamma \in \Gh$. So we may consider $M$ as a $\Gamma$-module. Being an element in $P$-proj, $f$ has the left lifting property with respect to the maps $L^{\Gamma}(M,n)\to K^{\Gamma}(M,n+1)$ and $K^{\Gamma}(M,n) \to B\Gamma$ in $\hSh/B\Gamma$. By adjunction (\ref{adjoint}), this shows that the maps
\begin{equation}\label{fZC}
\begin{array}{c}
f^{\ast}:Z_{\Gamma}^n(E\Gamma \times_{B\Gamma} Y;M)\to Z_{\Gamma}^n(E\Gamma \times_{B\Gamma} X;M)~\mathrm{and}\\
C_{\Gamma}^n(E\Gamma \times_{B\Gamma} Y;M)\to C_{\Gamma}^n(E\Gamma \times_{B\Gamma} X;M)\times_{Z_{\Gamma}^{n+1}(E\Gamma \times_{B\Gamma} X;M)}Z_{\Gamma}^{n+1}(E\Gamma \times_{B\Gamma} Y;M) \end{array}
\end{equation}
are surjective. By the definition of cohomology with local coefficients, this shows that $f$ also induces an isomorphism $H^n(f;\Mh)$ for every $n\geq 0$ and every finite abelian coefficient system $\Mh$ over $Y$. By Proposition 2.11 of \cite{profinhom}, this implies that $f$ is a weak equivalence.\\
4. The remaining point is to show that $W\cap Q$-proj $\subseteq P$-proj. So let $f:X\to Y$ be a map in $Q$-proj that is also weak equivalence. We deduce on the one hand that $f^*:C^n_{\Gamma}(Y;M) \to C^n_{\Gamma}(X;M)$ is surjective for all $n\geq 0$ and all finite abelian $\Gamma$-modules $M$ in $\Gh$. Moreover, since $f$ is a weak equivalence, this implies that the induced maps in (\ref{fZC}) are surjective for all such $n$ and $M$. The adjunction (\ref{adjoint}) then yields the desired lifting property of $f$ with respect to all maps $L^{\Gamma}(M,n)\to K^{\Gamma}(M,n+1)$ and $K^{\Gamma}(M,n) \to \ast$ in $\hSh$. The left lifting property with respect to the remaining maps in $P$ follows by an analogous argument, see also the proof of Lemma 2.8 of \cite{profinhom}.\\ 
This shows that we have found a fibrantly generated model structure on $\hSh$. Since the maps in $Q$-proj include the monomorphisms, every object in $\hSh$ is cofibrant which implies that the model structure is left proper. That the cofibrations are exactly the monomorphisms in $\hShg$ has been shown in Lemma 2.13 of \cite{profinhom}.
\end{proof}
We consider the category $\Sh$ of simplicial sets with the model structure of \cite{homalg} and denote its homotopy category by $\Hh$. For the proof of the following proposition, we refer again to \cite{profinhom}.
\begin{prop}\label{adjcompletion}
1. The level-wise completion functor $\compl: \Sh \to \hSh$ preserves weak equivalences and cofibrations.\\
2. The forgetful functor $|\cdot|:\hSh \to \Sh$ preserves fibrations and weak equivalences between fibrant objects.\\
3. The induced completion functor $\compl: \Hh \to \hHh$ and the right derived functor $R|\cdot|:\hHh \to \Hh$ form a pair of adjoint functors.
\end{prop}
\begin{defn}\label{defhomotopygroups}
Let $X$ be a pointed profinite space and let $RX$ be a fibrant replacement of $X$ in the above model structure on $\hSh$. Then we define the {\em $n$th profinite homotopy group of $X$} for $n \geq 2$ to be the profinite group
$$\pi_n(X):=\pi_0(\Omega^n(RX)).$$
\end{defn}
One should note, that to be a fibration in $\hSh$ is a stronger condition than in $\Sh$.  The profinite structure of the $\pi_nX$, would not be obtained by taking homotopy groups for $|X| \in \Sh$. For example, the fundamental group of the simplicial finite set $S^1$ as an object in $\hSh$ is equal to $\hat{\Z}$.
\begin{remark}
Morel \cite{ensprofin} proved that there is a model structure on $\hSh$ for each prime number $p$ in which the weak equivalences are maps that induce isomorphisms on $\Z/p$-cohomology. The fibrant replacement functor $R^p$ yields a rigid version of Bousfield-Kan $\Z/p$-completion. The homotopy groups for this structure are pro-$p$-groups being defined in the same way as above using $R^p$ instead of $R$. 
\end{remark}
\subsection{Profinite $G$-spaces}
Let $G$ be a fixed profinite group. Let $X$ be a profinite set on which $G$ acts continuously, i.e. there is a continuous map $\mu:G \times X \to X$ satisfying the usual axioms of group operation. In this situation we say that $X$ is a profinite $G$-set.\\
If $X$ is a profinite space and $G$ acts continuously on each $X_n$ such that the action is compatible with the structure maps, then we call $X$ a profinite $G$-space. We denote by $\hShg$ the category of profinite $G$-spaces with level-wise continuous $G$-equivariant morphisms. For an open and hence closed normal subgroup $U$ of $G$, let $X_U$ be the quotient space under the action by $U$, i.e. the quotient $X/\sim$ with $x \sim y$ in $X$ if both are in the same orbit under $U$. The following lemma is the analogue of the characterization of discrete spaces with a profinite group action.  
\begin{lemma}\label{actionlemma}
Let $G$ be a profinite group and $X$ a profinite space with a $G$-action. Then $X$ is a profinite $G$-space if and only if the canonical map $\phi:X \to \lim_U X_U$ is an isomorphism, where $U$ runs through the open normal subgroups of $G$. 
\end{lemma}
\begin{remark}
Let us quickly sketch the construction of colimits in $\hSh$ and $\hShg$. So let $\{X_i\}_{i \in I}$ be a diagram of profinite spaces. Let $X$ be the colimit of the underlying diagram of spaces, i.e. $X:=\colim_i |X_i|$, and let $\varphi_i:X_i \to X$ be the canonical maps in $\Sh$. We define a set $\Rh$ of equivalence relations on $X$, which are simplicial subsets of $X\times X$, to be the set of simplicial equivalence relations $R$ on $X$ such that \\
- $X_n/R_n$ is finite in each degree $n$, i.e. $X/R$ is a simplicial finite set,\\
- $\varphi_i^{-1}(R)$ is open in each $X_i\times X_i$ for all $i$, i.e. $\varphi_i^{-1}(R_n)$ is an open subset in each $X_{i,n}\times X_{i,n}$.\\
Then $\Rh$ is filtered from below and we define the colimit of the diagram in $\hSh$ to be the completion of $X$ with respect to $\Rh$, i.e. $\Xh:=\lim_{R\in \Rh} X/R$ in $\hSh$.
It is equipped with a canonical map $X \stackrel{\iota}{\to} \Xh$ which sends each $x\in X$ to the sequence of its equivalence classes $[x]_R \in X/R$. The image of $\iota $ is dense in $\Xh$. There are canonical maps $\phi_i:X_i \stackrel{\varphi_i}{\to} X \stackrel{\iota}{\to} \Xh$ in $\hSh$ that provide $\Xh$ with the universal property of a colimit in $\hSh$.\\
If $\{X_i\}_{i \in I}$ is a diagram of profinite $G$-spaces, then we modify $\Rh$ to $\Rh_G$ by the additional condition that every $R\in \Rh_G$ is in addition a $G$-invariant subspace of $X\times X$. Then we define the colimit of the diagram to be the completion of the underlying colimit with respect to $\Rh_G$.
\end{remark}
For $X$ and $Y$ in $\hSh$, the simplicial mapping space $\hom(X, Y)$  is defined in degree $n$ as the set of continuous maps $\Hom_{\hSh}(X \times \Delta[n], Y)$. If $G$ is a finite discrete group, considered as a constant simplicial profinite set, $\hom(G, Y)$ has a natural profinite structure induced by the profinite structure on $Y$. In order to show that $\hShg$ has the structure of a model category, we would like to use a right adjoint functor to the forgetful functor $\hShg \to \hSh$. But the problem is, that if $G$ is an arbitrary profinite group, the natural candidate for the right adjoint $\hom(G, Y)$ does not have to be a profinite space. Hence the usual adjoint functor argument using $\hom(G,-)$ may not be used to show that there is a model structure on $\hShg$. Initially, this forced us to consider an intermediate structure as in \cite{goerss} and then deduce Theorem \ref{Gmodelunstable} below via Bousfield localization. In this intermediate structure the cofibrations in $\hShg$ remained the monomorphisms. But it turns out that, in contrast to our initial belief, there is no model structure on $\hShg$ for which the cofibrations remain the monomorphisms and weak equivalences the maps that are weak equivalences in $\hSh$. But the referee suggested to prove Theorem \ref{Gmodelunstable} directly as we will do now. We are very grateful for the encouragement of the referee to look for a direct argument.\\
We want to show that for an arbitrary fixed profinite group $G$, the category $\hShg$ has a fibrantly generated model structure in which the weak equivalences are the maps whose underlying maps in $\hSh$ are weak equivalences. Therefore, we will define as before generating sets of fibrations and trivial fibrations. We modify the sets $\Th$, $\Gh$ and $\Lh$ defined above by allowing $G$-actions on their elements as follows. For every natural number $n\geq 0$ we consider the discrete sets $\{0,1, \ldots, n-1\}$ with a continuous $G$-action. As these sets are finite, each of them has only finitely many automorphisms. Hence this collection of $G$-sets forms a set which we denote by $\Thg$. Furthermore, we consider the finite discrete groups with underlying sets $\{0,1,\ldots,n-1\}$ as above with a continuous $G$-action. Again, as there are only finitely many relations and $G$-actions, this collection of finite $G$-groups forms a set which we denote by $\Ghg$. Finally, we denote by $\Lhg$ the collection of finite abelian groups $M \in \Ghg$ with an action of a finite group $\Gamma \in \Ghg$ which is compatible with the $G$-actions. Since $\Ghg$ is a set, it follows that $\Lhg$ forms a set as well.\\
Now let $P_G$ and $Q_G$ be the following two sets of morphisms:
$$\begin{array}{lll}
P_G & \mathrm{consisting~of} & E\Gamma \to B\Gamma, B\Gamma \to \ast,~ L^{\Gamma}(M,n)\to K^{\Gamma}(M,n+1),\\
   &   &  K^{\Gamma}(M,n)\to B\Gamma,~K(S,0)\to K(S,0)\times K(S,0)\\
   &  & K(S,0) \to \ast~ \mathrm{for~every~finite~set}~S\in \Thg,\\
    &  & \mathrm{every~abelian}~M\in \Lhg,~\Gamma \in \Ghg, ~\mathrm{and~every}~n\geq 0;\\
Q_G & \mathrm{consisting~of} & E\Gamma \to \ast,~ L^{\Gamma}(M,n) \to B\Gamma~ \mathrm{for~every}~\Gamma \in \Ghg,\\
& &  \mathrm{every}~M\in \Lhg ~\mathrm{and~every}~ n\geq 0. 
\end{array}$$
We call a morphism in $\hShg$ a weak equivalence (respectively fibration) if it is a weak equivalence (respectively fibration) in $\hSh$; and we call it a cofibration if it has the right lifting property with respect to all trivial fibrations in $\hShg$.
\begin{theorem}\label{Gmodelunstable}
These classes of maps define the structure of a left proper fibrantly generated simplicial model category on the category of profinite $G$-spaces with $P_G$ as a set of generating fibrations and $Q_G$ as a set of generating trivial fibrations. We denote its homotopy category by $\hHhg$. 
\end{theorem}
\begin{proof}
We will denote by $C^n_{\Gamma,G}(E\Gamma \times Y;M):=\Hom_{\hEh_{\Gamma,G}}(\Gamma^{n+1}\times Y_n, M)$ the $G$-$\Gamma$-equivariant cochains for $Y\in \hShg$, $\Gamma \in \Ghg$ and a $G$-$\Gamma$-group $M$; and similarly, $Z^n_G(Y;M)$ will denote the subgroup of $G$-equivariant cocycles. The strategy for the proof consists again in checking the four conditions of the dual of Kan's Theorem 11.3.1 in \cite{hirsch}. It is clear that the weak equivalences satisfy the 2-out-of-3 property and are closed under retracts.\\ 
1. Again the nontrivial cases to check are that $B\Gamma$, $K^{\Gamma}(M,n)$ and $K(S,0)\times K(S,0)$ are cosmall relative to $P_G$-cocell for $\Gamma\in \Ghg$, $M\in \Lhg$ and $S\in \Thg$. But this follows from a similar argument as in the proof of Theorem \ref{modelstructure} using the fact that $\hEh_{\Gamma,G}$ is equivalent to the pro-category of finite $G$-$\Gamma$-sets. \\ 
2. We have to show that every $Q_G$-fibration is both a $P_G$-fibration and a weak equivalence. We know from the proof of Theorem \ref{modelstructure} that the maps in $Q_G$ are trivial fibrations in $\hSh$ after forgetting the $G$-action. Since trivial fibrations are characterized by a lifting property, this implies that the maps in $Q_G$-cocell, which are limits of pullbacks of maps in $Q_G$, are also trivial fibrations. By \cite{hoveybook}, Theorem 2.1.19, this shows that the $Q_G$-fibrations are weak equivalences and $P_G$-fibrations.\\
3. The same argument as for Theorem \ref{modelstructure} shows that $P_G$-proj $\subset$ $Q_G$-proj. It remains to show that every $f:X\to Y$ in $P_G$-proj is a weak equivalence. Let $U \subset G$ be an open normal subgroup of $G$. For a finite set $T\in \Th$, a finite group $H\in \Gh$ and an abelian finite $H$-group $N\in \Lh$ we set  $S:=T^{G/U}$, $\Gamma:=H^{G/U}$ and $M:=N^{G/U}$ with $G/U$-action given by permuting the factors; thus we get $S\in \Thg$, $\Gamma \in\Ghg$ and $M\in \Lhg$. The crucial point is that $G$-equivariant maps from a profinite $G$-space $Z$ to $S$ are in one-to-one correspondence to maps in $\hSh$ from $Z/U$ after forgetting the $G$-action to $T$, and similarly for $\Gamma$ and $M$. Hence the argument of the proof of Theorem \ref{modelstructure} applied to the special maps in $P_G$ built from such $S$, $\Gamma$ and $M$ shows that $f/U$ is a weak equivalence in $\hSh$ for every open normal $U$ in $G$. By Lemma \ref{actionlemma} we know $f=\lim_U f/U$ and by \cite{profinhom}, Proposition 2.14, this implies that $f$ is a weak equivalence.\\  
4. The last point is to show that $W\cap Q_G$-proj $\subseteq P_G$-proj. So let $f:X\to Y$ be a map in $Q_G$-proj that is also weak equivalence. We deduce on the one hand that $f^*:C^n_{\Gamma,G}(E\Gamma \times Y;M) \to C^n_{\Gamma,G}(E\Gamma \times X;M)$ is surjective for all $n\geq 0$ and all finite abelian $G$-$\Gamma$-modules $M$ in $\Gh_G$. This shows that every $G$-equivariant cocycle in $Z_{\Gamma,G}^n(E\Gamma \times_{B\Gamma} X;M)$ has a lift to $G$-equivariant cochains. The fact that $f$ is a weak equivalence implies that this lift is actually also a cocycle in $Z_{\Gamma,G}^n(E\Gamma \times_{B\Gamma} Y;M)$.\\
Moreover, if we have an element $(\alpha,\beta)$ in the fibre product 
$$C_{\Gamma,G}^n(E\Gamma \times_{B\Gamma} X;M)\times_{Z_{\Gamma,G}^{n+1}(E\Gamma \times_{B\Gamma} X;M)}Z_{\Gamma,G}^{n+1}(E\Gamma \times_{B\Gamma} Y;M)$$ 
then there is a lift $\gamma_0$ of $\alpha$ in $C_{\Gamma,G}^n(E\Gamma \times_{B\Gamma} Y;M)$. Using the same special types of  $S$, $\Gamma$ and $M$ as in the previous item, we deduce from the proof of Theorem \ref{modelstructure} that $f/U$ is a trivial cofibration in $\hSh$ and hence so is $f=\lim_U f/U$. Thus $f$ is in $P$-proj after forgetting the $G$-action and there is also a possibly non-$G$-equivariant lift $\gamma_1$ of $(\alpha,\beta)$ in $C_{\Gamma}^n(E\Gamma \times_{B\Gamma} Y;M)$. But again since $f$ is a weak equivalence, it follows easily that the difference of $\gamma_0$ and $\gamma_1$ vanishes in $Z_{\Gamma,G}^{n+1}(E\Gamma \times_{B\Gamma} Y;M)$. Hence $\gamma_0$ is a $G$-equivariant lift of $(\alpha,\beta)$. Applying the same argument for the lifting property with respect to $E\Gamma \to \ast$, we conclude that $f$ is in $P_G$-proj.\\
This shows that we have found a fibrantly generated model structure on $\hShg$. The left properness follows from the same property for $\hSh$. 
\end{proof}

\begin{cor}
A map $f:X\to Y$ is a cofibration in $\hShg$ if and only if it is a level-wise monomorphism and $G$ acts freely on each set $Y_n -f(X_n)$. In particular, an object $X$ in $\hShg$ is cofibrant if and only if $G$ acts freely on each profinite set $X_n$.
\end{cor}
\begin{proof}
The previous theorem implies that the cofibrations in $\hShg$ are the maps in $Q_G$-proj. We recall that a map $X\to L(M,n)$ in $\hShg$ corresponds to a map $X_n \to M$ in $\hEhg$. So if a map $f:X\to Y$ is in $Q_G$-proj, then the left lifting property with respect to the maps $L(M,n)\to \ast$ in $Q_G$ corresponds to the existence of a lift $Y_n\to M$ for any map $X_n \to M$ in $\hEhg$. And similarly, maps $X\to E\Gamma$ in $\hShg$ correspond to maps $X_0 \to \Gamma$ in $\hEhg$. By the universality of limits, this implies that $f:X\to Y$ is in $Q_G$-proj if and only if it has the left lifting property with respect to maps $E\Gamma \to \ast$ and $L(M,n)\to \ast$ for all profinite $G$-groups $\Gamma$ and all profinite $G$-modules $M$ and $n\geq 0$. Now let $f:X\to Y$ be a cofibration. For any given $n\geq 0$, we choose $M$ to be the free profinite abelian group on $X_n$ with $G$-action induced by the action of $G$ on $X_n$ and the obvious injection $X_n \to M$ which corresponds to a map $\alpha:X\to L(M,n)$ in $\hShg$. Hence if there is a lift of $\alpha$ to $Y$, $f$ must be level-wise an injection. Moreover, if $N$ denotes a profinite group with a free $G$-action, we replace $M$ by $M\times N$ and see that the action of $G$ on $Y_n - f(X_n)$ must be free, since $f$ can be extended to a map $Y_n \to M\times N$ in $\hEhg$. Here we use the fact that a projective profinite $G$-set $Z$ is in fact free, since $Z/U$ is a projective and hence free $G/U$-set for all open normal subgroups $U$ of $G$ and $Z=\lim_U Z/U$.\\
Now let $f$ be a level-wise injection with a free $G$-action on $Y_n - f(X_n)$ for every $n\geq 0$. In order to show that $f$ is a cofibration, we need to show that for an injective map $f:X\to Y$ in $\hEhg$, $G$ acting freely on $Y - f(X)$, and every map $\alpha:X\to \Gamma$ in $\hEhg$ from $X$ to a profinite $G$-group $\Gamma$, there is a lift $\beta:Y\to \Gamma$ of $\alpha$. Now, if $f$ is injective, so is $f/G$. Since profinite groups are injective objects in $\hEh$ by \cite{profinhom}, Lemma 2.7, we know that $\alpha/G$ has a lift $\beta_0:Y/G \to \Gamma/G$ for $f/G$. But since the action of $G$ on $Y-f(X)$ is free, we can lift $\beta_0$ to a $G$-equivariant map $\beta:Y \to \Gamma$ such that $\alpha=\beta \circ f$ in $\hEhg$.  
\end{proof}

We equip the category $\hSh/BG$ of profinite spaces over $BG$ with the usual model structure in which a map is a weak equivalence, cofibration or fibration if and only if it is a weak equivalence, cofibration or fibration in $\hSh$. We have frequently used adjunction (\ref{adjoint}) and the pair of adjoint functors $F:\hSh/BG \to \hShg$, $X\mapsto EG\times_{BG}X$, and $U=(-)_{hG}:\hShg \to \hSh/BG$, $Y\mapsto EG\times_G Y$. Now we can say more about this adjunction, see \cite{hoveybook} for the terminology of Quillen equivalences.
\begin{cor}
The two functors $F$ and $U$ form a pair of Quillen equivalences between the category of profinite spaces over $BG$ and the category of profinite $G$-spaces.  
\end{cor}
\begin{proof}
We have already seen that $F$ is left adjoint to $U$. Moreover, we observe from the previous corollary that $F$ preserves cofibrations and trivial cofibrations. So $(F,U)$ forms a Quillen pair of adjoint functors. The remaining point is to observe that the map $EG\times_{BG}(EG\times_G Y)=EG\times Y\to Y$ is a weak equivalence, since $EG\to \ast$ is a trivial fibration. By \cite{hoveybook}, Corollary 1.3.16, this implies the result.
\end{proof}
\begin{defn}
Let $X$ be a profinite $G$-space and $M$ a profinite $G$-module. We define the $G$-equivariant cohomology of $X$ with coefficients in $M$ to be
$$H_G^n(X;M):=\Hom_{\hHhg}(X,K(M,n)).$$
\end{defn}
\begin{remark}\label{remZ/p}
1. Let $p$ be any prime number. The method to prove Theorem \ref{Gmodelunstable} also applies to Morel's $\Z/p$-model structure on $\hSh$ of \cite{ensprofin} and the action of a profinite group $G$. More generally, given a set $L$ of primes, one can use the proof of Theorem \ref{Gmodelunstable} to show that there is a model structure on $\hShg$ in which the weak equivalences are maps that induce isomorphisms with respect to pro-$L$-fundamental groups and continuous cohomology with local coefficient systems in finite $L$-groups. We just have to adapt the classes $P_G$ and $Q_G$ by requiring that the groups $\Gamma$ and $M$ are all finite $L$-groups.\\
2. In the proof of the theorem we have already used the following fact. Let $f:X\to Y$ be a map in $\hShg$ such that the underlying map of $f/U:X/U \to Y/U$ is a weak equivalence in $\hSh$ for every open normal subgroup $U$ of $G$. The homotopy invariance of the limit functor in $\hSh$ shown in \cite{profinhom}, Proposition 2.23, then implies that $f$ is a weak equivalence as well.
\end{remark}
\subsection{Homotopy fixed points and homotopy orbits}
We define the homotopy fixed points as usually as the function space of continuous maps coming from $EG$. 
%
\begin{defn}\label{homotopyfixedpoints} 
Let $G$ be a profinite group, let $X$ be a profinite $G$-space and let $X\mapsto RX$ be a fixed functorial fibrant replacement  in $\hShg$. We define the profinite homotopy fixed point space of $X$ to be the space of $G$-invariant continuous maps from $EG$ to $RX$:
$$X^{hG}:=\hom_{\hShg}(EG, RX).$$
\end{defn}
\begin{prop}\label{prop3.4}
Let $M$ be a profinite $G$-module. Then the homotopy groups of the simplicial set $K(M,n)^{hG}$ are equal to the continuous cohomology of $G$, i.e. for $0\leq k \leq n$ we have
$$\pi_kK(M,n)^{hG}=H^{n-k}(G;M).$$
\end{prop}
\begin{proof}
By definition of the cohomology $H^{n-k}(G;M)$ via homogeneous continuous cochains, there is an isomorphism $\pi_0\homg(EG,K(M,n))=H^n(G;M)$. The above adjointness (\ref{adjoint}) induces an isomorphism $\pi_0\hom_{\hSh/BG}(BG,K^G(M,n))=H^n(G;M)$. Now, applying the functor $\hom_{\hSh/BG}(BG,-)$ to the homotopy fibre square 
$$\xymatrix{
K^G(M,n) \ar[d] \ar[r] & BG \ar[d] \\
BG \ar[r] & K^G(M,n+1) }$$
shows that $\hom_{\hSh/BG}(BG,K^G(M,n))$ is homotopy equivalent to the loop space $\Omega\hom_{\hSh/BG}(BG,K^G(M,n+1))$. Hence $\pi_k\hom_{\hSh/BG}(BG,K^G(M,n))=H^{n-k}(G;M)$.
\end{proof}

One should note that, as indicated in the formulation of the proposition, there is a little subtlety about homotopy fixed points for profinite spaces. For an arbitrary profinite group $G$ and a fibrant profinite space $X$, $X^{hG}=\homg(EG, X)$ is in general not a profinite space. Nevertheless, $X^{hG}$ is an interesting object for studying actions of profinite groups.\\
The situation is different for example in the case of a $p$-adic analytic profinite group $G$. So for a moment we suppose that $G$ is a $p$-adic analytic group and that $M$ is a profinite $\Z_p[[G]]$-module, being the inverse limit $M=\lim_{\alpha}M_{\alpha}$ of finite $G$-modules $M_{\alpha}$. Let $X=K(M,n)$ be a fibrant Eilenberg MacLane space in $\hSh$. By Proposition \ref{prop3.4},  we know that $\pi_kK(M,n)^{hG}=H^{n-k}(G;M)$.  Moreover, since $G$ is $p$-adic analytic, it has an open normal subgroup which is a Poincar\'e pro-$p$-group. This implies that $H^n(G;M_{\alpha})$ is finite for each $\alpha$ and that  $H^n(G;M)=\lim_{\alpha}H^n(G;M_{\alpha})$. This shows that in this case $X^{hG}$ has itself a natural profinite structure.\\
Back to an arbitrary profinite group $G$. A profinite $G$-space $X \in \hShg$ may be considered as a functor from $G$ as a groupoid to $\hSh$. From this point of view, $\homg(\ast, X)=X^G$ is the limit of this functor in $\hSh$. Moreover, for $X\in \hShg$ fibrant, we can consider $\homg(EG,X)$ as the homotopy limit in $\Sh$.\\ 
As we mentioned earlier, in contrast to our initial belief, there is no simplicial model structure on $\hShg$ in which the weak equivalences are as above but the cofibrations are exactly all monomorphisms. To show this, let us assume there was such a structure. Then any map $\ast \to EG$ from the point to $EG$ would be a trivial cofibration. For $X$ fibrant in $\hShg$, it would induce a weak equivalence $f:\hom(EG,X) \to \hom(\ast,X)$. But since $X$ is fibrant in $\hShg$, the limit $X^G$ would still be fibrant in $\hSh$. But the homotopy groups of $X^G$ would then be profinite groups. So if $f$ is a weak equivalence of spaces, the homotopy groups of $\hom_G(EG,X)$ would have to be profinite as well. But using Proposition \ref{prop3.4}, it is easy to find a counterexample such that the homotopy groups are not profinite.  
So Proposition \ref{prop3.4} and the following theorem show that $\homg(EG,X) \in \Sh$, although not a profinite space, is the object we are interested in here.   
\begin{theorem}\label{fixeddescent}
Let $G$ be a profinite group and let $X$ be a pointed profinite $G$-space. Assume either that $G$ has finite cohomological dimension or that $X$ has only finitely many nonzero homotopy groups. Then there is a convergent descent spectral sequence for the homotopy groups of the homotopy fixed point space starting from continuous cohomology with profinite coefficients: 
$$E_2^{s,t}=H^s(G;\pi_t(X)) \Rightarrow \pi_{t-s}(X^{hG}).$$
\end{theorem}
\begin{remark}
As in \cite{goerss} 4.9, one should note that this is a second quadrant homotopy spectral sequence whose differentials go
$$d_r:E_r^{s,t}\longrightarrow E_r^{s+r,t+r-1}.$$
Moreover, $E_r^{s,t}$ is not defined for $t-s<0$ and $E_r^{s,s}$ can only receive differentials. In  \cite{goerss} 4.9, Goerss calls such a deformed spectral sequence fringed along the line $t=s$. Complete convergence of such an object is defined by Bousfield and Kan in \cite{bouskan} IX \S 5.3.
\end{remark}
\begin{proof}
This is a version of the homotopy limit spectral sequence of Bousfield and Kan for profinite spaces. We consider the category $c\Sh$ of cosimplicial spaces equipped with the model structure of \cite{bouskan} X, \S 4. There is the cosimplicial replacement functor given in codimension $n$ by $\homg(G^n,X) \in \Sh$, the simplicial space of continuous $G$-equivariant maps. If $X$ is fibrant in $\hShg$, its cosimplicial resolution is a fibrant object in $c\Sh$. Now define the total space of a cosimplicial space $Y$ to be $$\Tot Y:=\lim_s \Tot_s Y$$
where $\Tot_s Y:=\hom (\sk_s\Delta[\cdot],Y)$ in which $\sk_s\Delta[\cdot]$ is the $s$-skeleton of the cosimplicial standard simplex and $\hom$ denotes the usual function space. Then there is a spectral sequence of the cosimplicial replacement of $X$ which is the spectral sequence associated to the tower of fibrations that arises from the total space of the cosimplicial replacement of $X$. We have to check that the $E_2$-term is continuous cohomology of $G$.
By an analogue of \cite{bouskan} X, 7.2, there are natural isomorphisms 
$$E_2^{s,t}\cong \pi^s\pi_t(\homg(G^*,X))$$ 
for $t\geq s \geq 0$, where the right-hand side denotes the $s$th cohomotopy of the cosimplicial group $\pi_t(\homg(G^*,X))$. Since $\homg(G^n,X)$ is fibrant, there are natural isomorphisms of cosimplicial groups $\pi_t\homg(G^*,X)\cong \Hom_{\hEh_G}(G^*,\pi_t X)$ as remarked in \cite{bouskan} XI, 5.7. This implies that the above cohomotopy are cohomology groups of the complex $C^{\ast}(G;\pi_tX)$ given in degree $s$ by the set of continuous maps from $G^s \to \pi_tX$. If $\pi_1X$ is not abelian, this also holds for $s=0,1$, where $H^s(G;\pi_1X)$ is still a pointed set. Hence we have identified the $E_2$-term with the continuous cohomology groups of the statement.\\
It follows from the definition of $\homg(G^*,X)$ that the total space of this cosimplicial object is equal to $\homg(EG,X)$, i.e. the abutment of the spectral sequence is $\pi_{t-s}X^{hG}$. Finally, the assumptions imply $\lim_r^1 E_r^{s,t}$ vanishes and the spectral sequence converges. completely, cf. \cite{bouskan} IX \S 5.3.
\end{proof}

We recall from \cite{profinhom} that the homology $H_{\ast}(X):=H_{\ast}(X;\hZ)$ of a profinite space $X$ is defined to be the homology of the complex $C_{\ast}(X)$ consisting in degree $n$ of the profinite groups $C_n(X):=\hFab (X_n)$, the free abelian profinite group on the profinite set $X_n$. The differentials $d$ are the alternating sums $\sum_{i=0}^{n}d_i$ of the face maps $d_i$ of $X$. If $M$ is a profinite abelian group, then $H_{\ast}(X;M)$ is defined to be the homology of the complex $C_{\ast}(X;M):=C_{\ast}(X)\hat{\otimes}M$, where $\hat{\otimes}$ denotes the completed tensor product, see e.g.~\cite{ribes} \S 5.5.\\ 
For $X \in \hShg$, the homotopy orbit space $X_{hG}:=EG\times_G X$ can be viewed as the homotopy colimit of the $G$-action on $X$. Moreover, the homology $H_s(X;M)$ is itself a profinite $G$-module for any profinite abelian group $M$. This gives rise to the following spectral sequence.
\begin{theorem}\label{homologicalorbitdescent}
Let $X$ be a profinite $G$-space and $M$ a profinite abelian group. There is a first quadrant homology spectral sequence for the homology groups of $X_{hG}$ starting from the continuous homology $H_t(X;M)$ of $G$ with coefficients in the profinite $G$-modules converging to the homology of the homotopy orbit space of $X$:
$$E^2_{s,t}=H_s(G;H_t(X;M)) \Rightarrow H_{s+t}(X_{hG};M).$$
\end{theorem}
\begin{proof}
This is a profinite version of the homotopy colimit spectral sequence of Bousfield and Kan \cite{bouskan}, XII \S 5.7. We can assume that $X$ is fibrant in $\hShg$. By \cite{bouskan}, XII \S 5.2, in order to calculate the homotopy colimit of a diagram, one can first take a simplicial resolution of this diagram. In our case this yields a simplicial profinite space $X\times G^{\ast}$, where, for every $k$, $G^{k}$ denotes the constant simplicial set of the $k$-fold product of $G$. 
The homotopy colimit is then equal to the diagonal of the bisimplicial resolution of $X$ induced by $G$, i.e. 
$$X_{hG}\cong \mathrm{diag}(X\times G^{\ast}) \in \hSh.$$ 
It follows immediately that, by applying homology, the bisimplicial profinite set yields a bisimplicial abelian group which has a profinite structure in each bilevel and in which the maps are continuous group homomorphisms. It is a standard argument to deduce from the bisimplicial abelian group a spectral sequence
$$E^2_{s,t}={\colim_G}^sH_t(X;M) \Rightarrow H_{s+t}(\hocolim_G X;M)$$ 
where $\colim^s_G$ denotes the $s$th left derived functor of the functor induced by the $G$-action. It remains to remark that, all groups being equipped with a natural profinite structure, $\colim^s_G$ is the derived functor of $\colim_G$ in the category of profinite $G$-modules; and that $\colim_G B$ is the orbit group $B/G$ of a profinite $G$-module $B$. Moreover, $H_s(G,B)$ is the $s$th left derived functor of the functor $B\mapsto B/G$ by \cite{ribes}, Proposition 6.3.4.   
\end{proof}

\subsection{Profinite $G$-spectra}
A profinite spectrum $X$ consists of a sequence $X_n \in \hShp$ of pointed profinite spaces for $n\geq0$ and maps $\sigma_n:S^1 \wedge X_n \to X_{n+1}$ in $\hShp$. A morphism $f:X \to Y$ of spectra consists of maps $f_n:X_n \to Y_n$ in $\hShp$ for $n\geq0$ such that $\sigma_n(1\wedge f_n)=f_{n+1}\sigma_n$. We denote by $\hSp$ the corresponding category of profinite spectra. By Theorem 2.36 of \cite{profinhom}, there is a stable homotopy category $\hSHh$ of profinite spectra. In this model structure, a map $f:X\to Y$ is a stable equivalence if it induces a weak equivalence of mapping spaces $\map(Y,E) \to \map(X,E)$ for all $\Omega$-spectra $E$; and $f$ is a cofibration if $X_0 \to Y_0$ and the induced maps $X_n \amalg_{S^1\wedge X_{n-1}}S^1 \wedge Y_{n-1} \to Y_n$ are monomorphisms for all $n$. \\
Now let $G$ be as always a profinite group. We consider the simplicial finite set $S^1$ as a profinite $G$-space with trivial action. 
\begin{defn}\label{Gspectra}
Let $S^1_G$ be a cofibrant replacement of $S^1$ in $\hShg$. We call $X$ a (naive) profinite $G$-spectrum if, for $n\geq 0$, each $X_n$ is a pointed profinite $G$-space and each $S_G^1\wedge X_n \to X_{n+1}$ is a $G$-equivariant map. We denote the category of profinite $G$-spectra by $\hSpg$.
\end{defn}
\begin{theorem}\label{modelGspectra}
There is a model structure on profinite $G$-spectra such that a map is a stable weak equivalence (resp.~fibration) if and only if it is a stable weak equivalence (resp.~fibration) in $\hSp$. The fibrations are the maps with the right lifting property with respect to maps that are weak equivalences and cofibrations. We denote its homotopy category by $\hSHhg$.
\end{theorem}
\begin{proof}
Starting with the model structure on $\hShg$ of Theorem \ref{Gmodelunstable}, the stable model structure is obtained in the same way as for $\hSh$ from the techniques of \cite{hovey} and the localization results of \cite{etalecob}, Theorems 6 and 14, for fibrantly generated model categories. It is also clear from this construction and Theorem \ref{Gmodelunstable} that a map in $\hSpg$ is a stable weak equivalence (resp.~fibration) if and only if it is a stable weak equivalence (resp.~fibration) in $\hSp$. 
\end{proof}

\begin{cor}
If $X$ is a profinite $G$-spectrum then each stable profinite homotopy group $\pi_kX$ is a profinite $G$-module. 
\end{cor}
\subsection{Homotopy orbit spectra}
Our aim is to construct a spectral sequence as above that starts with the continuous homology of $G$ with coefficients the profinite $G$-module $\pi_kX$ and that converges to the stable homotopy groups of the homotopy orbit spectrum $X_{hG}:=EG_+\wedge_G X$ of the $G$-action on $X$.\\ 
For this purpose, we consider the simplicial resolution of the diagram induced by the $G$-action on $X$. As in \cite{bouskan} XII, it is defined to be the simplicial profinite spectrum $X\wedge (G^{\ast})_+$. Here we denote again, for every $k$, by $(G^{k})_+$ the constant simplicial set of the $k$-fold product of $G$ as above but with an additional basepoint and in level $n$, $(X\wedge (G^{k})_+)_n:=X_n\wedge (G^k)_+ \in \hShp$. We may consider $X\wedge (G^{\ast})_+$ either as a bisimplicial profinite set or as a simplicial profinite spectrum. Since the diagonal functor $d$ from pointed bisimplicial profinite sets to pointed simplicial profinite sets commutes with smashing with $S^1$, we may apply the diagonal functor level-wise to get a spectrum $d(X\wedge (G^{\ast})_+)$ with $d(X\wedge (G^{\ast})_+)_n=\mathrm{diag}(X_n\wedge (G^{\ast})_+)$, cf. \cite{jardine} 4.3. The homotopy colimit is then isomorphic to the diagonal spectrum $d(X\wedge (G^{\ast})_+)$ of the simplicial resolution, i.e.  
$$X_{hG}\cong d(X\wedge (G^{\ast})_+) \in \hSp.$$ 
%
For a simplicial spectrum, Jardine shows in \cite{jardine} \S 4, how to construct a spectral sequence that computes the homotopy groups of the diagonal spectrum, Corollary 4.22 of \cite{jardine}:
\begin{equation}\label{4.22}
E^2_{s,t}=H_s(\pi_t(Y_{\ast})) \Rightarrow \pi_{s+t}(d(Y)).
\end{equation}
The whole construction can be applied in the category of simplicial profinite spectra. For a simplicial profinite spectrum $[k]\mapsto Y_k$, the stable homotopy group $\pi_t(Y_k)$ has a natural profinite structure and $\pi_t(Y_{\ast})$ becomes a simplicial profinite abelian group. The resulting complex has continuous differentials. Its homology groups also carry an induced natural profinite structure. Hence (\ref{4.22}) may be viewed as a spectral sequence in profinite abelian groups.\\
We use (\ref{4.22}) for two applications. The first one shows that the diagonal functor from simplicial profinite spectra to profinite spectra and hence the homotopy orbit functor respects weak equivalences.
\begin{prop}
Let $X\to Y$ be a map between simplicial profinite spectra such that, for each $n\geq 0$, the map $X_n \to Y_n$ is a stable equivalence in $\hSp$. Then the induced map $d(X) \to d(Y)$ is a stable equivalence in $\hSp$.
\end{prop}
The second application of (\ref{4.22}) is what we were really looking for. Let $X$ be a profinite $G$-spectrum. Then the homology of the simplicial profinite abelian group $\pi_t(X\wedge (G^{\ast})_+)$ is just the continuous group homology $H_s(G,\pi_tX)$ with profinite coefficients $\pi_tX$. Hence we get the following result.
\begin{theorem}\label{homotopyorbitforspectra}
Let $G$ be a profinite group and let $X$ be a profinite $G$-spectrum. Then there is a convergent spectral sequence
$$E^2_{s,t}=H_s(G;\pi_t(X)) \Rightarrow \pi_{s+t}(X_{hG}).$$
\end{theorem}
A spectral sequence for the homotopy orbit spectrum under an action of a profinite group $G$ had already been studied in different contexts, in particular by Davis. In \cite{davis}, Davis considers discrete $G$-spectra and calls a spectrum $Y$ an $f$-spectrum if $\pi_qY$ is a finite group for each integer $q$. Let $G$ be a countably based profinite group and $Y_0 \leftarrow Y_1\leftarrow Y_2\leftarrow \ldots$ a tower of $G$-$f$-spectra such that the level-wise taken homotopy limit $Y=\holim_i Y_i$ is a $G$-spectrum. Then the homotopy groups of $Y$ are profinite groups. For this situation, Theorem 5.3 of \cite{davis} provides a spectral sequence as in Theorem \ref{homotopyorbitforspectra}. We remark that since $\pi_qY_i$ is finite, for each $Y_i$ there is some profinite $G$-spectrum $X_i$ which is fibrant in $\hSp$ such that its underlying spectrum, i.e. after forgetting the profinite structure, is weakly equivalent to $Y_i$. The homotopy limit $X$ is then also weakly equivalent to $Y$ and $Y_{hG}$ is weakly equivalent to $X_{hG}$. Hence each tower $Y$ of $G$-$f$-spectra may be considered as a profinite $G$-spectrum and the spectral sequence of Theorem 5.3 of \cite{davis} is a special case of the spectral sequence of Theorem \ref{homotopyorbitforspectra} above that arises naturally in the category of profinite $G$-spectra for an arbitrary profinite group $G$.
\section{Galois actions}
Now we return to our motivating examples for profinite spaces and continuous group actions of the introduction. The starting point for \'etale homotopy theory is the work of Artin and Mazur \cite{artinmazur}. The goal was to define invariants as in Algebraic Topology for a scheme $X$ that depend only on the \'etale topology of $X$. They associated to a scheme $X$ a pro-object in the homotopy category $\Hh$ of spaces. Friedlander rigidified the construction by associating to $X$ a pro-object in the category $\Sh$ of simplicial sets. The construction is technical and we refer the reader to \cite{fried} for any details, in particular for the category of rigid hypercoverings. As a reminder for the reader who is familiar with the techniques, the definition is the following: For a locally noetherian scheme $X$, the {\rm \'etale topological type of $X$} is the pro-simplicial set $\Et X := \mathrm{Re} \circ \pi :HRR(X) \to \Sh$ sending a rigid hypercovering $U_{\cdot}$ of $X$ to the simplicial set of connected components of $U_{\cdot}$. If $f: X \to Y$ is a map of locally noetherian schemes, then the strict map $\Et f:\Et X \to \Et Y$ is given by the functor $f^{\ast}:HRR(Y) \to HRR(X)$ and the natural transformation $\Et X \circ f^{\ast} \to \Et Y$.\\
In \cite{etalecob} and \cite{profinhom}, we studied a profinite version $\hEt$ of this functor by composing $\Et$ with the completion from pro-$\Sh$ to the category of simplicial profinite sets $\hSh$. The advantage of $\hEt X$ is that we have taken the limit over all hypercoverings in a controlled way and obtain an actual simplicial set that still remembers the continuous invariants of $X$. Let us summarize the key properties of $\hEt X$ in the following proposition, which is due to Artin-Mazur \cite{artinmazur} and Friedlander \cite{fried}, but one might also want to have a look at \cite{profinhom} for the comparison with continuous cohomology of Dwyer-Friedlander \cite{dwyfried} and Jannsen \cite{jannsen}. 
\begin{prop}\label{Etsummary}
1. Let $\bar{x}$ be a geometric point of $X$. It also determines a point in $\hEt X$. The profinite fundamental group $\pi_1(\hEt X, \bar{x})$ of $\hEt X$ as an object of $\hSh$ is isomorphic to the \'etale fundamental group $\pi^{\et}_1(X,\bar{x})$ of $X$ as a scheme.\\
2. Let $F$ be pro-object in the category of locally constant \'etale sheaf of finite abelian groups on $X$. It corresponds bijectively to a local coefficient system $F$ of profinite groups on $\hEt X$. Moreover, the cohomology of $\hEt X$ with profinite local coefficients in $F$ equals the continuous \'etale cohomology of $X$, i.e. $H^{\ast}(\hEt X,F)\cong H^{\ast}_{\mathrm{cont}}(X,F)$.
\end{prop}
\begin{remark}
1. One should note that the set of connected components of $\hEt X$ is equal to the profinite completion of the set of connected components of $X$. As we are usually in the situation that $X$ has only finitely many or even a single connected component, this is not a problem.\\
2. Continuous \'etale cohomology is a more sophisticated version of $\ell$-adic cohomology for schemes. Usually, $\ell$-adic cohomology $H^i(X;\Z_{\ell}(j))$ is defined as an inverse limit over $n$ of $H_{\et}^i(X;\Zln(j))$. If the groups $H_{\et}^i(X;\Zln(j))$ are not finite, $H^i(X;\Z_{\ell}(j))$ might not have all the good properties that one desires. Dwyer-Friedlander \cite{dwyfried} and Jannsen \cite{jannsen} gave a better behaved definition. The one in \cite{dwyfried} works for all locally constant profinite coefficients, the one in \cite{jannsen} yields a derived functor approach for more general inverse systems of coefficients but requires different conditions. The cohomology of $\hEt X$ gives a very direct construction of continuous \'etale cohomology for locally constant profinite coefficients. This explains in which sense $\hEt X$ remembers continuous invariants of $X$, cf. \cite{fried} \S 4. The different definitions of continuous cohomology agree when they are all defined. Moreover, if the groups $H^i(X;\Zln(j))$ are finite for all $n$, then $\ell$-adic and continuous cohomology with $\Z_{\ell}(j)$-coefficients agree. 
\end{remark}
Now let $k$ be a field, $\ok$ a separable closure of $k$ and $G_k:=\Gal(\ok/k)$. Let $X$ be a variety over $k$, i.e. a separated reduced and irreducible scheme of finite type over $k$, and let $\oX=X\otimes_k \ok$ be the base change of $X$ to $\ok$. 
Unfortunately, it is not clear if $\hEt \oX$ is always a profinite $G_k$-space in the above sense. The action of $G_k$ on the sets of connected components of rigid hypercovers might not be continuous in each level. Nevertheless, there is a canonical model for $\hEt \oX$ in $\hShg$.
\begin{lemma}\label{limit}
The canonical $G_k$-equivariant map $\alpha:\hEt \oX\longrightarrow \lim_L \hEt X_L$ is a weak equivalence, where the limit is taken over all finite Galois extensions $L/k$ in $\ok$.
\end{lemma}
\begin{proof}
Let $L/k$ be a finite Galois extension with Galois group $G_L$. Since fundamental groups of profinite spaces commute with limits, there is a canonical isomorphism
$$\pi_1(\lim_L \hEt X_L)\cong \lim_L \pi_1(\hEt X_L).$$
Moreover, by \cite{sga1} IX, \S 6, we know that $\pi_1(\hEt \oX)=\pi_1^{\et}(\oX)$ is isomorphic to $\lim_L \pi_1(\hEt X_L)$. This shows that $\alpha$ induces isomorphisms on fundamental groups.\\
It remains to show that $\alpha$ also induces isomorphisms on cohomology with local coefficients of finite abelian groups. This follows from the fact that $H_{\et}^{\ast}(\oX;F)$ is equal to the colimit $\colim_L H_{\et}^{\ast}(X_L;F_L)$ for any locally constant sheaf $F$ on $\oX$ whose pullback to $X_L$ is denoted by $F_L$. From the analogous equality $H^{\ast}(\lim_L \hEt X_L;F)=\colim_L H^{\ast}(\hEt X_L;F)$ and Proposition \ref{Etsummary}, we deduce that $\alpha$ is a weak equivalence.
\end{proof} 

Since the action of $G_k$ on $\hEt X_L$ factors through the finite group $\Gal(L/k)$, this action is continuous on the profinite space $\hEt X_L$. As $G_k$ is the limit of all the $\Gal(L/k)$, this shows that the action of $G_k$ on $\lim_L \hEt X_L$ is continuous, cf. \cite{bourb} III \S 7, No 1. We will use this profinite $G_k$-space as a continuous model for $\hEt \oX$ in $\hShg$ and will denote it by 
$$\chEt \oX:=\lim_L\hEt X_L.$$
\begin{remark}
This problem vanishes if $G_k$ is strongly complete, i.e. if it is isomorphic to the profinite completion of its underlying group $|G_k|$, or in other words, if every subgroup of finite index is open, see \cite{ribes}. In this case, the $G_k$-action on $\hEt \oX$ would be continuous for any variety $X$. The class of strongly complete profinite groups contains the class of all finitely generated profinite groups by the work of Nikolov and Segal \cite{niksegal}. For example the absolute Galois group of $p$-adic local fields are finitely generated, cf. \cite{jw}. The absolute Galois group of a number field is in general not strongly complete as subgroups of finite index which are not open be easily constructed in such groups.
\end{remark}
By \cite{fried} and \cite{profinhom}, we know that $\hEt k$ is homotopy equivalent to $BG_k$ and $\hEt \ok$ to $EG_k$ in $\hSh$. As mentioned in the introduction the natural sequence (\ref{fibersequence}) inspires us to think of $\hEt X$ as the homotopy orbit space of $\hEt \oX$, just as Cox showed for real algebraic varieties in \cite{cox} Theorem 1.1. The following theorem generalizes Cox's result to arbitrary fields. The point is that $\hEt \oX \to \hEt X$ is homotopy equivalent to a principal $G_k$-fibration, see \cite{profinhom}, p. 593.\footnote{In \cite{profinhom} and in a previous version of this paper, it was stated that this map is a principal fibration, which is only true up to homotopy. So one may find here a rigorous treatment of the problem.} 
\begin{theorem}\label{etalehomorbit}
Let $k$ be a field with absolute Galois group $G_k$ and let $X$ be a geometrically connected variety over $k$. Then the canonical map 
$$\varphi:\chEt \oX \times_{G_k} EG_k \to \hEt X$$
is a weak equivalence of profinite spaces.
\end{theorem}
\begin{proof}
By the definition of weak equivalences, we have to show that $\varphi$ induces an isomorphism on the profinite fundamental groups and on continuous cohomology with finite abelian coefficients systems. Let us start with the fundamental groups. We know from the work of Grothendieck \cite{sga1} IX, Th\'eor\`eme 6.1, that there is a short exact sequence
$$1 \longrightarrow \pi_1^{\et}(\oX,\bar{x}) \longrightarrow \pi_1^{\et}(X,x) \to G_k \longrightarrow 1$$ 
for every geometric point $\bar{x}$ of $\oX$ with image $x$ in $X$. On the other hand, the map of  profinite spaces $\chEt \oX \times EG_k \to \chEt \oX \times_{G_k}EG_k$ is a principal $G_k$-fibration by definition, see \cite{profinhom}. Hence it is also locally trivial, see e.g. \cite{gj} V, Lemma 2.5, and may be considered as a Galois covering with group $G_k$. By the classification of coverings of profinite spaces via the fundamental group in \cite{profinhom}, Corollary 2.3, we deduce that there is a similar short exact sequence for profinite spaces such that $\pi_1(\varphi)$ fits in a commutative diagram
$$
\xymatrix{
1 \ar[r] & \pi_1(\chEt \oX \times EG_k,\bar{x}) \ar[r] \ar[d] & \pi_1(\chEt \oX\times_{G_k} EG_k,x) \ar[r] \ar[d] & G_k \ar[r] \ar@{=}[d] & 1\\
1 \ar[r] & \pi_1^{\et}(\oX,\bar{x}) \ar[r] & \pi_1^{\et}(X,x) \ar[r] & G_k \ar[r] & 1}
$$
for every basepoint $\bar{x}$ of $\oX$. Since $EG_k$ is contractible, the left vertical arrow is an isomorphism and we conclude that $\varphi$ induces an isomorphism on fundamental groups.\\
To prove that $\varphi$ also induces an isomorphism on cohomology we apply two Serre spectral sequences. Let $F$ be a locally constant \'etale sheaf of finite abelian groups on $X$. On the one hand there is the Hochschild-Serre spectral sequence for \'etale cohomology starting from continuous cohomology of $G_k$ with coefficients in the discrete $G_k$-module $H_{\et}^t(\oX;F)$, cf. \cite{milne}: 
$$E_2^{s,t}=H^s(G_k;H_{\et}^t(\oX;F))\Rightarrow H_{\et}^{s+t}(X;F).$$
On the other hand the fibre sequence
$$\chEt \oX \longrightarrow \chEt \oX \times_{G_k} EG_k \longrightarrow BG_k$$
 induces a Serre spectral sequence 
$$E_2^{s,t}=H^s(G_k;H^t(\chEt \oX;F))\Rightarrow H^{s+t}(\chEt \oX\times_{G_k} EG_k;F)$$
where $F$ also denotes the associated local coefficient system on $\hEt \oX$ by Proposition \ref{Etsummary}. This spectral sequence may be constructed in the profinite setting just as in \cite{dress}, see also \cite{ensprofin} \S 1.5 and \cite{dehon} \S 1.5 for pro-$p$-versions. It remains to observe that there is a natural isomorphism between these spectral sequences which is compatible with $\varphi$ using the isomorphism $H_{\et}^t(\oX;F)\cong H^t(\hEt \oX;F)$. Since these groups vanish for $t>\dim\, \oX$, the two spectral sequences are strongly convergent which finishes the proof of the theorem.
\end{proof}

Hence we may consider $\hEt X$ as the homotopy orbit space of $\hEt X_{\ok}$ under its natural Galois action. We will use this key theorem for three applications. On the one hand we deduce Galois descent spectral sequences for \'etale (co)homology theories. The last application is a remark on Grothendieck's section conjecture for smooth proper curves of genus at least two over number fields. But first we show that we can lift this equivalence of the two points of view to the level of motivic spectra \cite{a1hom}. Over a fixed base field $k$, the stable homotopy category $\SHh(k)$ of motivic spectra can be obtained as follows. We start the category $\Delta^{\mathrm{op}}\mathrm{PreShv}(\Sm_k)$ of simplicial presheaves on $\Sm_k$, the category of quasi-projective smooth schemes over $k$, with the projective model structure, i.e. the weak equivalences (fibrations) are objectwise weak equivalences (fibrations) of simplicial sets. Then we localize this model structure with respect to coproducts, Nisnevich hypercovers and maps of the form $X\times \A^1 \to X$ for every $X\in \Sm_k$. This provides a model for motivic spaces. Then we stabilize this construction by considering $\Pro^1_k$-spectra, i.e. sequences of motivic spaces $E_n$ together with maps $\Pro^1_k \wedge E_n \to E_{n+1}$ for every $n\geq 0$, see e.g. \cite{etalecob} for this particular model. The resulting homotopy category of motivic spectra over $k$ is denoted by $\SHh(k)$.\\
We know that the \'etale realization functor above can be extended to a functor from motivic spectra  to the homotopy category of profinite spectra over $\hEt k$, cf. \cite{etalecob}, Theorem 31: 
$$\hEt: \SHh(k) \to \hSHh/\hEt k.$$ 
This extension can be achieved with the model structure on $\hSh$ and $\hSp$ of Theorem \ref{modelstructure} and Theorem 2.36 of \cite{profinhom}, respectively, if $\car k=0$. 
If $\car k=p >0$, we have to complete away from the characteristic by using the $L$-model structure on $\hSh$ and $\hSp$ for any set of primes $L$ with $p \notin L$, see Remark \ref{remZ/p} above. Now we remark that the adjointness discussed in the beginning of Section 2.3 of taking homotopy orbits and pullbacks via maps to $BG_k$ has an analogue for profinite spectra. This implies that we can reconstruct \'etale realization as a functor to the category of profinite $G_k$-spectra. Hence the following theorem is in this sense equivalent to Theorem 31 of \cite{etalecob}.
\begin{theorem}
Let $k$ be a field of characteristic zero and let $G_k$ be its absolute Galois group. The \'etale realization functor above defines a functor 
$$\hEt: \SHh(k) \to \hSHh_{G_k}$$
to the stable homotopy category of profinite $G_k$-spectra by sending a motivic spectrum $E$ to $\hEt E_{\ok}$, where $E_{\ok}$ is the base change of $E$ to $\ok$, with its natural $G_k$-action.\\
If $k$ has positive characteristic $p$, the same statement holds when we equip $\hSh$ and $\hSp$ with the $L$-model structure for any chosen set $L$ of primes not containing $p$.
\end{theorem}
\subsection{Galois descent}
As an application of the homotopy orbit spectral sequence we consider a variant of \'etale homotopy groups. For a pointed locally noetherian scheme, we define $\pi_{\ast}^{\et,s}X:=\pi_{\ast}^s(\Sigma^{\infty}\hEt X)$ to be the stable \'etale homotopy groups of $X$. The absolute Galois group $G_k$ acts continuously on each profinite group $\pi^{\et,s}_q(\oX)$. There is the following Galois descent spectral sequence for these groups by Theorem \ref{homotopyorbitforspectra}.
\begin{theorem}\label{stablehomotopyetaledescent}
Let $X$ be a geometrically connected variety over a field $k$ with absolute Galois group $G_k$. There is a convergent spectral sequence for the stable \'etale homotopy groups of $X$:  
$$E^2_{p,q}=H_p(G_k;\pi^{\et,s}_q(\oX)) \Rightarrow \pi^{\et,s}_{p+q}(X).$$
\end{theorem}
In the same way we get Galois descent spectral sequences for \'etale topological cohomology theories, e.g. \'etale cobordism \cite{etalecob}. Let $MU$ be the simplicial spectrum representing topological complex cobordism and let $\hMU$ be its profinite completion. In \cite{etalecob}, an \'etale topological version of cobordism for smooth schemes has been studied. It is the theory represented by $\hMU$ via $\hEt$, i.e. in degree $n$ we set 
$$
\hMU_{\et}^{n}(X):=\Hom_{\hSHh}(\Sigma^{\infty}\hEt X, \hMU[n])$$
where $\hMU[n]$ denotes the $n$th shift of $\hMU$ and $X$ is assumed to be a pointed scheme. We can reformulate this definition using function spectra and get an isomorphism
\begin{equation}\label{etalecobfunction}
\hMU_{\et}^{n}(X) \cong \pi_n \hom_{\hSp}(\Sigma^{\infty}\hEt X, \hMU).\end{equation}
Let us denote the function spectrum on the right-hand side of (\ref{etalecobfunction}) by 
$$\hMU_{\et}^X:=\hom_{\hSp}(\Sigma^{\infty}\hEt X, R\hMU)$$
where $R$ means a fibrant replacement in $\hSp$. This description and an analogue of Proposition \ref{etalehomorbit} implies that \'etale cobordism satisfies Galois descent in the following sense, generalizing \cite{dwyfried}, Proposition 7.1. In order to prove this, we start with the following lemma. 
\begin{lemma}\label{finitegaloisdescent}
For each finite Galois extension $L/k$ with Galois group $G_L=\Gal(L/k)$ and $X_L=X\otimes_k L$, there is a natural equivalence of simplicial spectra 
$$\hMU_{\et}^X \stackrel{\simeq}{\longrightarrow} (\hMU_{\et}^{X_L})^{hG_L}$$
where $G_L$ acts on $\hMU_{\et}^{X_L}$ via its induced action on $\hEt X_L$.
\end{lemma}
\begin{proof}
The assertion is implied by the following sequence of equivalences, where we omit $\Sigma^{\infty}$: 
$$\begin{array}{rcl}
\hom_{\hSp}(\hEt X, R\hMU) & \stackrel{\simeq}{\longrightarrow} & \hom_{\hSp}(\hEt X_L \times_{G_L} EG_L, R\hMU)\\
 & \stackrel{\simeq}{\longrightarrow} & \hom_{\Spg}(EG_L,\hom(\hEt X_L, R\hMU))\\
  & \stackrel{\simeq}{\longrightarrow} & (\hMU_{\et}^{X})^{hG_L}\end{array}$$
where the first equivalence follows from Theorem \ref{etalehomorbit}, the second follows from adjointness for the simplicial finite set $EG_L$ and the third one is the definition of homotopy fixed point spectra for finite groups acting on simplicial spectra.   
\end{proof}
\begin{theorem}\label{galoisdescent}
Let $k$ be a field with absolute Galois group $G_k$ of finite cohomological dimension and let $X$ be a geometrically connected pointed variety over $k$. There is a convergent spectral sequence 
$$E_2^{s,t}=H^s(G_k;\hMU_{\et}^{t}(\oX)) \Rightarrow \hMU_{\et}^{s+t}(X)$$
starting from continuous cohomology of $G_k$ with coefficients the discrete $G_k$-module $\hMU_{\et}^{\ast}(\oX)$. 
\end{theorem}
\begin{proof}
Each finite quotient $G_L$ of $G_k$ induces a finite Galois covering $X_L \to X$ which is homotopy equivalent to finite Galois covering of the profinite space $\hEt X$ using the argument in the proofs of Theorem \ref{etalehomorbit} and Lemma \ref{limit}. The well-known homotopy fixed point spectral sequence for finite groups acting on simplicial spectra together with the Lemma \ref{finitegaloisdescent} yield a spectral sequence
$$E_2^{s,t}=H^s(G_L;\hMU_{\et}^{t}(X_L)) \Rightarrow \hMU_{\et}^{s+t}(X)$$ 
for every $i$. Now the weak equivalence $\hEt \oX \simeq \lim_L \hEt X_L$ of Lemma \ref{limit} implies $\hMU_{\et}^{t}(\oX)\cong \colim_L \hMU_{\et}^{t}(X_L)$ and hence there is an isomorphism $$H^s(G_k;\hMU_{\et}^{t}(\oX))\cong \colim_L H^s(G_L;\hMU_{\et}^{t}(X_L)).$$ Since spectral sequences commute with colimits, this implies the assertion of the theorem.
\end{proof}

The homological counterpart, called \'etale bordism, is defined as
$$\hMU^{\et}_{n}(X):=\Hom_{\hSHh}(S^n,\Sigma^{\infty}\hEt X\wedge \hMU).$$ 
In this case, the descent spectral sequence for \'etale bordism has a more direct construction as the homotopy orbit spectral sequence of a generalized homology theory as in Theorem \ref{homologicalorbitdescent} above.
\begin{theorem}\label{homologyetaledescent}
Let $k$ be a field with absolute Galois group $G_k$ and let $X$ be a geometrically connected pointed variety over $k$. There is a convergent spectral sequence for the \'etale bordism of $X$:  
$$E^2_{s,t}=H_s(G_k;\hMU^{\et}_{t}(\oX)) \Rightarrow \hMU^{\et}_{s+t}(X).$$
\end{theorem}
\subsection{A remark on Grothendieck's section conjecture}
We conclude with an application of the developed theory of homotopy fixed points to Galois actions.\footnote{I would like to thank Kirsten Wickelgren for interesting discussions about this topic.}  Let us briefly recall the statement of Grothendieck's section conjecture \cite{grothendieck}. It is part of a much more general picture drawn by Grothendieck in \cite{grothendieck} which predicts that, for some class of varieties over $k$, the functor of taking fundamental groups should be in some sense fully faithful. Detailed accounts on the conjecture can be found e.g. in \cite{mochizuki1}, \cite{kim} and \cite{stix}. Let $k$ be a field and $G_k$ its absolute Galois group. Let $X$ be a geometrically connected variety over $k$. We have already used that the functoriality of $\pi_1=\pi_1^{\et}$ induces a short exact sequence
$$
\xymatrix{1 \ar[r] & \pi_1\oX \ar[r] & \pi_1X\ar[r] & G_k \ar[r] & 1}
$$
where we omit the basepoints for this discussion. Another application of the functoriality of $\pi_1$ shows that every $k$-rational point $a\in X(k)$ induces a section $s_a:G_k \to \pi_1X$ which is well-defined up to conjugacy by $\pi_1\oX$. 
The section conjecture of Grothendieck's in \cite{grothendieck} predicts that this map has an inverse.
\begin{conjecture}\label{sectionconjecture} {\rm (Grothendieck)}
Let $k$ be a field which is finitely generated over $\Q$ and let $X$ be a smooth, projective curve of genus at least two. The map $a\mapsto s_a$ is a bijection between the set $X(k)$ of $k$-rational points of $X$ and the set of $\pi_1\oX$-conjugacy classes of sections $G_k \to \pi_1X$.
\end{conjecture}
It is well-known that the map $a\mapsto s_a$ is injective. The hard part is the surjectivity. The varieties that the Grothendieck conjecture is about, especially the curves of Conjecture \ref{sectionconjecture}, are so called $K(\pi,1)$-varieties. And here is the point where \'etale homotopy enters the stage. In terms of \'etale homotopy theory a variety $X$ is a $K(\pi,1)$-variety if the profinite universal covering space of $\hEt X$ is contractible. Or in other words, $\hEt X$ is weakly equivalent in $\hSh$ to the profinite classifying space $B\pi_1X$. Just as for spaces, it is also known for profinite spaces that there is a bijection between the set of homotopy classes of continuous maps of Eilenberg-MacLane spaces $\Hom_{\hHh}(K(G,1) \to K(\pi,1))$ and the set of outer continuous group homomorphisms $\Hom_{\mathrm{out}}(G,\pi)$. In light of the previous discussion this shows there is a bijection
\begin{equation}\label{scashomotopy}
\Hom_{\hHh/\hEt k}(\hEt k,\hEt X) \cong \Hom_{\mathrm{out},G_k}(G_k,\pi_1X),
\end{equation}
where the right-hand side denotes outer homomorphisms that are compatible with the projection to $G_k$. So Conjecture \ref{sectionconjecture} may be restated in the way that $\hEt$ is a fully faithful functor from $k$-rational points to homotopy classes of maps from $\hEt k$ to $\hEt X$.\\ 
All this is of course just a reformulation. But the point we want to stress is that the machinery of Galois actions developed above provides an interesting point of view for Conjecture \ref{sectionconjecture}. From the natural adjunction (\ref{adjoint}) induced by taking homotopy orbits  and the canonical homotopy equivalence $\hEt k \simeq BG_k$ we deduce by Theorem \ref{etalehomorbit} that there is a further canonical bijection
\begin{equation}\label{scashf}
\Hom_{\hHh/\hEt k}(\hEt k,\hEt X) \cong \Hom_{\hHh_{G_k}}(EG_k,\chEt \oX)\cong \pi_0(\chEt \oX)^{hG_k}\end{equation}
and hence in order to prove the section conjecture one could try to prove that the induced map
$$X(k)\longrightarrow \pi_0(\chEt \oX)^{hG_k}$$
is a bijection.\\
The slight shift of the point of view is very interesting, since the section conjecture over the real numbers $\R$ and topological analogues of it could be proved by Pal in \cite{pal1} using fixed and homotopy fixed point methods for finite groups, see also \cite{kirsten}. Further studies in this direction have been done by Pal in \cite{pal2}. The new input of this paper consists in Theorem \ref{etalehomorbit} and in the rigorous framework for homotopy fixed points for \'etale homotopy types for varieties over any base field which had been missing so far.\\
What one would like to do now is to factor the map
\begin{equation}\label{sectionhomfixedpoints}
X(k) \to \pi_0(\chEt \oX)^{hG_k}\end{equation}
and one would like to factor this map through some fixed point set under the $G_k$-action. Then the geometric part of the problem would be to show that $X(k)$ is isomorphic to this fixed point set. The homotopy theoretical part would be to show that the fixed point set is isomorphic to the homotopy fixed point set in the right-hand side of (\ref{sectionhomfixedpoints}). This might be possible by transferring comparison results for finite groups to the case of the profinite groups $G_k$ acting on the profinite space $\chEt \oX$, cf. \cite{carlsson}.  
\bibliographystyle{amsplain}

Mathematisches Institut, Universit\"at M\"unster, Einsteinstr. 62, D-48149 M\"unster\\
E-mail address: gquick@math.uni-muenster.de\\
Homepage: www.math.uni-muenster.de/u/gquick
\end{document}